\documentclass[a4paper,11pt]{amsart}
\usepackage[T1]{fontenc}
\usepackage[utf8]{inputenc}
\usepackage{indentfirst}
\usepackage{fancyhdr}
\usepackage{amssymb}
\usepackage{amsmath}
\usepackage{amsthm}
\usepackage{latexsym}
\usepackage{eucal}
\usepackage{pdfpages}
\usepackage{setspace}
\usepackage{braket}
\usepackage{version}
\usepackage{verbatim}
\usepackage{bm,bbm}
\usepackage{datetime}
\usepackage{mathtools}
\usepackage[normalem]{ulem}
\usepackage{soul}
\usepackage{cancel}
\usepackage{hyperref}
\hypersetup{
	colorlinks=true,
	linkcolor=black,
	citecolor=black,
}
\usepackage{cleveref}

\usepackage[a4paper,top=3.5cm,bottom=3.5cm,left=3cm,right=3cm]{geometry}
\DeclareMathAlphabet{\mathcal}{OMS}{cmsy}{m}{n}

\usepackage[symbol]{footmisc}

\numberwithin{equation}{section}

\theoremstyle{plain}
\newtheorem{theorem}{Theorem}[section]
\newtheorem{lemma}[theorem]{Lemma}
\newtheorem{proposition}[theorem]{Proposition}
\newtheorem{corollary}[theorem]{Corollary}

\theoremstyle{definition}
\newtheorem{remark}[theorem]{Remark}

\newtheorem*{namedthm*}{\namedthmname}

\newcommand{\pa}{\partial}
\newcommand{\de}{\partial}
\newcommand{\om}{\omega}
\newcommand{\sqom}{\sqrt{\omega}}

\newcommand{\norm}[1]{\left\lVert #1 \right\rVert}

\def \R{\mathbb R}

\def \N{\mathbb N}

\def \O{\mathcal O}

\title[Convergence of discrete solutions to the WKE]{On the convergence rates of discrete solutions to the Wave Kinetic Equation}

\author[M. Dolce]{Michele Dolce}
\address{Institute of Mathematics, EPFL, Station 8, 1015 Lausanne, Switzerland}
\email{michele.dolce@epfl.ch}

\author[R. Grande]{Ricardo Grande}
\address{International School for Advanced Studies (SISSA), Via Bonomea 265, 34136, Trieste, Italy}
\email{rgrandei@sissa.it}


\begin{document}
	\maketitle
  \begin{center}
  \vspace{-.5cm}
     \emph{``Dedicated to Pierangelo Marcati on the occasion of his 70th birthday"}
 \end{center}
	\begin{abstract}
	In this paper, we consider the long-term behavior of some special solutions to the Wave Kinetic Equation (WKE). This equation provides a mesoscopic description  of wave systems interacting  nonlinearly via the cubic NLS equation. Escobedo and Vel\'azquez showed that, starting with initial data given by countably many Dirac masses, solutions remain a linear combination of countably many Dirac masses at all times. Moreover, there is convergence to a single Dirac mass at long times. The first goal of this paper is to give quantitative rates for the speed of said convergence. In order to study the optimality of the bounds we obtain, we introduce and analyze a toy model accounting only for the leading order quadratic interactions.
	\end{abstract}
%\setcounter{tocdepth}{3}
%\tableofcontents
\section{Introduction}

In recent years, there has been an increasing interest in understanding the average behavior of out-of-equilibrium systems of many waves undergoing weakly nonlinear interactions. A fundamental example of such a system is given by the cubic Schr\"odinger equation. 

The kinetic formalism for such wave system, known as wave kinetic theory, consists in studying the evolution of the variance of the Fourier coefficients of such wave systems in the kinetic limit (i.e. as their size grows and the strength of the interactions diminishes), see \cite{nazarenko2011wave} for details. This variance, upon rescaling in time, has been shown to satisfy the Wave Kinetic Equation (WKE):
\begin{equation}\label{eq:WKE_noniso}
\pa_t n (t,\xi)  = \mathcal{K} (n(t,\cdot )), \qquad \xi\in\R^3,
\end{equation}
where 
\[
\mathcal{K}(n)(\xi) =\int_{\substack{(\R^3)^3\\ \{\xi=\xi_1-\xi_2+\xi_3\}}}\delta_{\R}(|\xi_1|^2 - |\xi_2|^2+|\xi_3|^2-|\xi|^2)\, n_1 n_2 n_3 n \left( \frac{1}{n} - \frac{1}{n_1} + \frac{1}{n_2} - \frac{1}{n_3} \right) \,  d\xi_1\, d\xi_3 
\]
with $n_j = n(\xi_j)$.

%These systems are generally governed by a large number of microscopic physical laws, and due to their complexity one would like to derive a way to understand their average behavior from a \emph{mesoscopic} viewpoint. This is precisely the goal of kinetic theory: deriving a kinetic equation that describes the typical behavior of such a system.

Kinetic equations for wave systems first appeared in the work of Peierls \cite{Peierls}, Nordheim \cite{nordheim1928kinetic} and in the work of Hasselman in the context of water waves \cite{Hasselmann1,Hasselmann2}. A rigorous mathematical derivation was only recently achieved, starting with the work of Buckmaster, Germain, Hani, Shatah \cite{BGHS}, that of Collot and Germain \cite{CollotGermain2}, and culminating with the recent works of Deng and Hani \cite{DengHani,DengHani2,DengHani3,DengHani4,DengHani5}, where a full derivation is obtained. Other wave systems have also recently been considered, see for instance \cite{StaffilaniTran,HRST}.

Despite the rigorous justification of the WKE, many questions remain unanswered regarding the behavior of solutions to \eqref{eq:WKE_noniso}. The study of the well-posedness and long-term behavior of certain solutions to the WKE \eqref{eq:WKE_noniso} was initiated by Escobedo and Vel\'azquez in \cite{escobedo2015theory}, as well as in \cite{GIT}. In their work, they consider radial initial data in 3D, where one can explicitly integrate the delta function in \eqref{eq:WKE_noniso} in the angular variables, leading to the \emph{isotropic} WKE:
\begin{equation}
	\label{eq:WKEstrong}
\begin{split}
		\de_t g_1&=\int_{D(\omega_1)} \Phi \left[ \left(\frac{g_1}{\sqrt{\om_1}}+\frac{g_2}{\sqrt{\om_2}}\right)\frac{g_3g_4}{\sqrt{\om_3\om_4}}-\left(\frac{g_3}{\sqrt{\om_3}}+\frac{g_4}{\sqrt{\om_4}}\right)\frac{g_1g_2}{\sqrt{\om_1\om_2}} \right]d\om_2d\om_3d\om_4,\\
		g_1 \mid_{t=0} & =g^{in}(\om_1), \qquad g_i d\om_i=g(t,d\om_i),\\
		\Phi&=\min\{\sqom_1,\sqom_2,\sqom_3,\sqom_4\},\\
		D(\om_1)&=\{\om_3\geq 0,\om_4\geq 0; \om_3+\om_4\geq \om_1\}, \qquad \om_1\geq 0.
\end{split}
\end{equation} 
Here $g(t,\omega)=|\xi|\, n(t,|\xi|^2)$ and $\omega = |\xi|^2$, where $n$ is the solution to \eqref{eq:WKE_noniso}. It is convenient to work with $g$ so that it can be interpreted as a density of particles in the space $\{\omega\geq 0\}$ \cite{escobedo2015theory}. One can thus define 
\[
\text{Mass:} \quad M=\int_{\mathbb{R}_+}g(t,d \omega), \qquad \text{Energy:} \quad E=\int_{\mathbb{R}_+}\omega g(t,d \omega).
\]
Both quantities above are conserved as long as they are both initially finite.

Escobedo and Vel\'azquez then consider the weak formulation of the equation \eqref{eq:WKEstrong}, namely
\begin{equation}\label{eq:WKE}
\begin{split}
		\frac{d}{dt} \left( \int_{\R_{+}} \varphi(t,\om) g(t,d\om) \right)  = & \ \int_{\R} \partial_t \varphi(t,\omega)\, g(t,d\omega) \\
	 & + \int_{\R_{+}^3} \Phi \, \frac{g_1 g_2 g_3}{\sqrt{\om_1 \om_2 \om_3}} \,[\varphi_4 +\varphi_3 - \varphi_2 -\varphi_2]\, d\om_1 d\om_2 d\om_3,\\
		\om_4  = & \ \om_1+\om_2-\om_3\geq 0, \qquad \varphi_i=\varphi(\omega_i),
\end{split}
\end{equation}
almost everywhere for any test function $\varphi\in C^2_c ([0,T)\times \R_{+})$.

One may  show that \eqref{eq:WKE} is globally well-posed in a space of Radon measures $\mathcal{M}_{\rho}$ defined in \eqref{eq:defMrho} below.
%This space of measures allows some of the singular power-like solutions that we mentioned before, notably the Rayleigh-Jeans solution or the Kolmogorov-Zakharov spectra with constant flux of mass but misses  the spectra with constant flux of energy.
Moreover, it is possible to study the long-term behavior of such measure solutions and, among other results, Escobedo--Vel\'azquez prove that \cite{escobedo2015theory}:
\begin{itemize}
	\item If the (conserved) mass $M=\int_{\mathbb{R}^+} g^{in}d\omega$ is finite, then 
	\begin{equation}\label{eq:conv_delta}
	g(t,\cdot)\rightharpoonup \delta_{R^*} \qquad \mbox{as}\ t\rightarrow\infty,
	\end{equation}
	 with $R^*=\mathrm{inf} A^*$ where the set $A^*$ is defined in \eqref{eq:an}-\eqref{eq:a*}.
	\item If $\mathrm{supp}(g^{in})\subseteq \mathbb{N}$, then 
	\begin{equation}\label{eq:disc_dynamics}
	\mathrm{supp}(g(t,\cdot))\subseteq\mathbb{N} \qquad \mbox{for any}\ t>0.
	\end{equation}
\end{itemize}

The goal of this paper is to quantify the speed of the convergence in \eqref{eq:conv_delta} in a special case where we start with energy concentrated on a set of discrete frequencies away from the origin. In this setting, we first observe an instantaneous spreading of energy towards all (discrete) frequencies, before the solution converges to a single Dirac mass concentrated at $R^*$. These complicated dynamics were first studied qualitatively in \cite{escobedo2015theory} in a more general scenario. The purpose of this article is to present some quantitative results in the special case of initial data displaying discrete dynamics.

\subsection{Statement of results}

We consider the case $\mathrm{supp}(g^{in})\subseteq \mathbb{N}$. It is therefore useful to introduce a set of test functions $\varphi_n \in C^\infty_c(\mathbb{R}_+)$ such that $\mathrm{supp}(\varphi)\subset B(n,1/2)$ and define 
 \begin{equation}
	\label{def:Fn}
	F_n (t):= \int_{\R_{+}} \varphi_n (\om ) g(t,d\om) .
\end{equation}
Since $\mathrm{supp}(g(t,\cdot))\subseteq \mathbb{N}$ by \eqref{eq:disc_dynamics}, the functions $F_n$ fully describes the dynamics of $g$ solving \eqref{eq:WKE}. By a direct inspection of the weak formulation \eqref{eq:WKE}, we will derive the (infinite) system of ODE's describing the evolution of $F_n$. The particular structure of the equations is not relevant for the statement of the main result, but the equations of interest can be found in Lemma \ref{lem:tedious_equations}. Our first result, proved in Section \ref{sec:main}, is the following:
\begin{theorem}
	\label{th:main}
	 Let $g^{in}=M_1\delta_1+\sum_{j=2}^{\infty} m_j\delta_j$ be the initial data of \eqref{eq:WKEstrong} with $(m_j)_{j=2}^{\infty}\in \ell^{1,r}(\R_{+})$ with $r>1$ (see  \eqref{eq:l11}). Define  
  \[
  M_2:=\sum_{j=2}^{\infty} m_j, \qquad E=\sum_{j=1}^{\infty} j\, m_j =: M_1+E_2,
  \] 
  and, without loss of generality, assume $M_1+M_2=1$.
  Then, there exists $t_0>0$ such that for any $t> t_0$ the following inequalities hold true: 
\begin{align}\label{bd:F2}
 \frac{c_2}{t-t_0+3E/b_1(t_0)+C_2}\leq F_2 (t) \leq 1-F_1(t)\leq \frac{c_1(t_0)}{\sqrt{b_1(t_0)(t-t_0)+3E}},
	\end{align} 
where $c_1(t_0),\, b_1(t_0)$ are given in \eqref{eq:c1} whereas $c_2,C_2$ are explicitly computable. Moreover, if $M_1\geq 3E_2/19$ then $t_0=0$. 
\end{theorem}
Notice that, being the mass conserved, we have $F_k(t)\leq \sum_{j=2}^{\infty} F_j(t)=1-F_1(t)$ for all $k\geq 2$. Therefore, the upper bound in Theorem \ref{th:main} is true for all $F_k$ with $k\geq 2$, whereas we are only able to prove the lower bound for $F_2$.

Thanks to the result above, we know that, if we wait long enough or we start with $M_1$ large enough, the convergence towards the Dirac mass at $\{1\}$ is at least $\mathcal{O}(t^{-1/2})$ but cannot be faster than $\mathcal{O}(t^{-1})$.
\begin{remark}
If we set $t_0=0$ but $M_1<3E_2/19$ then we need to change the lower bound in \eqref{bd:F2} to $\O (t^{-\alpha})$ with a time-rate $\alpha>1$, explicitly given in \eqref{def:alpha}. This suggests that the rate of convergence could be faster for small times and slow down as mass concentrates at $\{1\}$. Note that the convergence is at most polynomial for all times.
\end{remark}

\begin{remark}
    If the lower bound for $1-F_1$ in \eqref{bd:F2} were sharp, then one could derive lower and upper bounds for the speed of convergence of all the functions $F_n$ for $n$ large enough. In such case, they would decay as $\O_n (t^{-1})$. This is the content of \Cref{prop:cond} in \Cref{sec:main}.
\end{remark}

In order to understand whether there exist solutions exhibiting a decay that saturate the lower bound in \eqref{bd:F2}, we propose a toy model where we only keep the terms in \eqref{eq:WKE} involving at least one interaction with the leading term $F_1$. Moreover, we replace all terms $F_1$ by its limit $1$. In Section \ref{sec:toy_model}, we show that these reductions give rise to the following quadratic toy model:
\begin{equation}\label{eq:intro_toy_model}
\frac{d}{dt}F_n = 4\,  \frac{F_n}{\sqrt{n}} \, \frac{F_{2n-1}}{\sqrt{2n-1}} - 4\, \frac{F_n}{\sqrt{n}}\, \sum_{k=2}^{n-1} \frac{F_k}{\sqrt{k}}- 2 \left(\frac{F_n}{\sqrt{n}}\right)^2+ 2\sum_{k=n+1}^{\infty} \frac{F_kF_{k+1-n}}{\sqrt{k(k+1-n)}}.
\end{equation}
We then look for self-similar solutions of the form
\begin{equation}\label{eq:intro_ansatz}
F_n(t) = \beta_n\, \frac{\sqrt{n}}{t}, \qquad \text{for } \beta_n\in [0,\infty),\ 2\leq n\in \N.
\end{equation}
Analogous \emph{ansatzs} in the continuous setting are common in the literature, see for instance the related works of Kierkels and Vel\'azquez \cite{KV,KV2} in the WKE context.

In such a toy model, positivity of solutions cannot be expected to hold anymore, unlike in \eqref{eq:WKE}. In fact, the question of existence of solutions \eqref{eq:intro_ansatz} with strictly positive $\beta_n$ remains an interesting open question, which we answer with a further reduction. Indeed, our next result consists of showing the existence of strictly positive solutions for a \textit{truncated} version of \eqref{eq:intro_toy_model}, since positivity is a key and physical feature of the solutions to the full WKE \eqref{eq:WKEstrong}. In \Cref{sec:toy_model}, we show that simple truncations of \eqref{eq:intro_toy_model} do not admit positive solutions. We thus  consider:
\begin{equation}\label{eq:intro_truncation}
    -\sqrt{n}\, \beta_n = 4\, \beta_n \, \beta_{2n-1} - 4\, \beta_n \, \sum_{k=2}^{n-1}\beta_k - 2 \beta_n^2 + 2\sum_{k=n+1}^{3N-2} \beta_k \beta_{k+1-n} \qquad \mbox{for}\ n=2,\ldots, N,
\end{equation}
where $N\in\N$, $N\geq 2$, may be as large as desired. We then have the following

\begin{theorem}\label{thm:toy_model}
 Fix any $N\in\N$ with $N\geq 4$. Fix $\lambda_1, \lambda_2>0$ such that 
    $\lambda_1\lambda_2 > N/4$.
Then there exists some $\delta_0 (N,\lambda_1, \lambda_2)>0$ such that for all $\delta<\delta_0$, there exists a solution to \eqref{eq:intro_truncation} (which solves \eqref{eq:intro_toy_model} for $n\leq N$ with the ansatz \eqref{eq:intro_ansatz}) such that 
\begin{equation}
    \begin{split}
        \beta_2 & = \sqrt{\lambda_1\lambda_2+ 1/8}+ \sqrt{2}/4 + \O ( \sqrt{N} \,\delta\, \max\{ \lambda_1,\lambda_2\}),\\
         \beta_{2N} & = \lambda_1 + \O (\delta),\\
         \beta_{2N+1} & = \lambda_2 + \O (\delta),\\
        0<\beta_j & =  \O ( \sqrt{N} \,\delta\, \max\{ \lambda_1,\lambda_2\}) \qquad j=3,\ldots,N,\\
        0<\beta_j & =  \O ( \delta) \qquad j=2N+2,\ldots,3N-2,\\
         \beta_{j} & = 0 \qquad \mbox{otherwise}.
    \end{split}
\end{equation}
\end{theorem}
The result in the theorem above is an example of a solution whose leading order terms are $\beta_2,\beta_{2N},\beta_{2N+1}$. This suggests that there could be self-similar solutions to \eqref{eq:intro_toy_model} of the form \eqref{eq:intro_ansatz}. In fact, our long-term goal is not only to construct such solutions on the toy model \eqref{eq:intro_toy_model}, but to carry out  a nonlinear perturbative argument for the full \eqref{eq:WKE} around this ``approximate'' solution found in Theorem \ref{thm:toy_model}.
%\red{\begin{itemize}
%	\item Comment in a nice way that the rates one may deduce for Escobedo-Velazquez are not correct due to a small typo \blue{I did it just before the proof of Prop 3.1, maybe we can copy it here?}
%\end{itemize}}

%\subsection{Organization}\red{we can remove organization, just write after the results where we prove them}
%
%This note is organized as follows. In \Cref{sec:previous} we summarize some previous work that is relevant to this paper. \Cref{sec:main} contains the proof of \Cref{th:main}, as well as the main results of this note. Finally, in \Cref{sec:appendix} we derive the system of equations that capture the dynamics of the WKE in the case of discretely supported initial data.

\subsection{Outline}

The article is organized as follows. In \Cref{sec:system}, we discuss some background results and present the discrete WKE associated to \eqref{eq:WKE} with initial data given by a linear combination of Dirac masses. In \Cref{sec:main} we prove \Cref{th:main}. In \Cref{sec:toy_model}, we introduce our toy model, discuss possible truncations and prove \Cref{thm:toy_model}. Finally, in \Cref{sec:appendix}, we give the derivation the discrete WKE from \eqref{eq:WKE}.

\subsection{Notation}

The set of natural numbers $\N$ is taken without $0$, namely $\N=\{1,2,\dots\}$, and we define $\R_{+}=[0,\infty)$. When the index set in a sum consists only of non-positive indices, we consider that sum to be zero, e.g. $\sum_{k=1}^{n-1} a_k=0$ whenever $n\leq 1$. 

%\textcolor{blue}{We need to define our space of test functions $\varphi \in C^2_c ([0,T)\times\mathbb{R}_+)$.}

We let $\ell^{1,r}(\R_{+})$, $r\geq 0$, be the Banach space of sequences $(m_j)_{j=1}^{\infty}$, $m_j\geq 0$, with the norm:
\begin{equation}\label{eq:l11}
\sum_{j=1}^{\infty} j^r\,m_j <\infty.
\end{equation}

Finally, we consider the space $\mathcal{M}_{\rho}$ of non-negative Radon measures $\mu$ such that
\begin{equation}\label{eq:defMrho}
 \norm{\mu}_{\rho}=\sup_{R>1} \frac{1}{(1+R)^{\rho}}\, \frac{1}{R}\,\int_{R/2}^R \mu(d\omega) + \int_0^1 \mu(d\omega) <\infty.
 \end{equation}
In this note we will consider initial data $\mu_0$ in $\mathcal{M}_{\rho}$ with some $\rho<-2$, which guarantees a finite and conserved energy. This is equivalent to requiring $(m_j)_{j\in\N}\in\ell^{1,r}(\R_{+})$ with $r>1$, as stated in \Cref{th:main}.

We will often omit the differential in some integrals when it is clear from the context, e.g.
\[
\int_0^{\infty} \mu(t,d\om)=\int_0^{\infty} \mu(t).
\]

\section{Previous results and discrete WKE}\label{sec:system}

In this section, we summarize a few results in \cite{escobedo2015theory} which will be useful in the rest of the paper. Furthermore, we present the equations satisfied by $F_n$ defined in \eqref{def:Fn}. A full proof of the derivation, which is technically simple yet computationally tedious, is given in \Cref{sec:appendix}.

First of all, by \cite[Proposition 2.28]{escobedo2015theory} we know that if $g^{in}\in \mathcal{M}_{\rho}$ with $\rho<-2$ then the weak solution to \eqref{eq:WKEstrong} has conserved and finite energy and mass for all times. As mentioned after \eqref{eq:defMrho}, we know that our initial data in Theorem \ref{th:main} is in $\mathcal{M}_\rho$ with $\rho<-2$ and therefore we always have finite and conserved mass and energy. 

Then we need two key results that will allow us to prove bounds on $F_1(t)$, and they are the foundation of our analysis for $F_n (t)$ for $n>1$. The first one is  a combination of  \cite[Proposition 2.22]{escobedo2015theory} and \cite[Lemma 2.25]{escobedo2015theory}.

\begin{lemma}\label{thm:monotonicity} 
	Suppose that $g\in \mathcal{M}_{\rho}$.
 %and let $\mathcal{S}^3$ be the group of permutations of three elements $\{1,2,3\}$. 
 Then, for any $\varphi\in C_b^2 (\R_{+})$,
\begin{equation}
	\label{eq:mono}
\int_{\R_{+}^3} \Phi\, \left[ \left(\frac{g_1}{\sqrt{\om_1}}+\frac{g_2}{\sqrt{\om_2}}\right)\frac{g_3g_4}{\sqrt{\om_3\om_4}}-\left(\frac{g_3}{\sqrt{\om_3}}+\frac{g_4}{\sqrt{\om_4}}\right)\frac{g_1g_2}{\sqrt{\om_1\om_2}} \right] \, \varphi_1 
= \int_{\R_{+}^3} \frac{g_1g_2g_3}{\sqrt{\om_1\om_2\om_3}} \mathcal{G}_{\varphi}
\end{equation}
where both integrals are in $d\omega_1d\omega_2d\omega_3$ and the following notation is used:
\begin{align}
 \label{def:G}\mathcal{G}_{\varphi} (\om_1,\om_2,\om_3) & = \frac{1}{3} \, \left[\sqrt{\om_{-}} \mathcal{H}^1_{\varphi}(\om_1,\om_2,\om_3) + \sqrt{(\om_0+\om_{-}-\om_{+})_{+}} \, \mathcal{H}^2_{\varphi} (\om_1,\om_2,\om_3)\right],\\
\label{def:H2}\mathcal{H}_{\varphi}^1 (\om_1,\om_2,\om_3)& =\varphi(\om_{+}+\om_{0}-\om_{-}) +\varphi(\om_{-}+\om_{+}-\om_{0}) - 2\varphi(\om_{+}),\\
\label{def:H1}\mathcal{H}_{\varphi}^2 (\om_1,\om_2,\om_3)& =\varphi(\om_{+}) + \varphi(\om_{-}+\om_{0}-\om_{+}) - \varphi(\om_{0}) - \varphi(\om_{-}),
\end{align}
%\begin{align}
%\mathcal{G}_{\varphi}  = \mathcal{G}_{\varphi} (\om_1,\om_2,\om_3) & = \frac{1}{6} \, \sum_{p\in\mathcal{S}^3} \mathcal{H}_{\varphi} (\om_{p(1)},\om_{p(2)},\om_{p(3)}) \, \Phi (\om_{p(1)},\om_{p(2)},\om_{p(3)})\\
%\mathcal{H}_{\varphi} (\om_1,\om_2,\om_3)& =\varphi(\om_3) + \varphi(\om_1+\om_2-\om_3) - \varphi(\om_1) - \varphi(\om_2).
%\end{align}
and 
\begin{align*}
\om_{+} (\om_1,\om_2,\om_3 ) & = \max \{ \om_1,\om_2,\om_3\},\qquad \om_{-} (\om_1,\om_2,\om_3 )  = \min \{ \om_1,\om_2,\om_3\},\\
\om_{0} (\om_1,\om_2,\om_3 ) & = \{ \om_1,\om_2,\om_3\}-\{ \om_{+},\om_{-}\}.
\end{align*}
Moreover, if $\varphi$ is convex we have that $\mathcal{G}_{\varphi}\geq 0$. 

Let $g\in \mathcal{M}_\rho$ be a weak solution to \eqref{eq:WKE} and $\varphi\in C(\R_{+})$ be a convex function. Then 
\begin{equation}\label{eq:convexity}
 \frac{d}{dt} \left( \int_0^{\infty} \varphi(\om) g(t,d\om) \right) \geq 0 \qquad \mbox{for a.e.}\ t\geq 0.
\end{equation}
\end{lemma}
The nice \textit{monotonicity} formula \eqref{eq:mono} is a direct computation using the symmetries of the equation \eqref{eq:WKE}. The proof can be found in \cite[Proposition 2.22]{escobedo2015theory}.

The next result guarantees that if we consider initial data $g^{in}$ with discrete support, the dynamics will be discrete and nontrivial for later times. Before we state these results, we define some auxiliary sets to identify $\mathrm{supp}(g(t))$, with $g$ being the solution to \eqref{eq:WKE}. Let $A_1 = \mbox{supp}(g^{in})$, define $A_n$ inductively as:
\begin{equation}\label{eq:an}
A_{n+1} = (A_n +A_n -A_n)\cap (0,\infty).
\end{equation}
The idea behind these sets is the following: the wave interactions captured by the WKE \eqref{eq:WKE} are those between waves of different frequencies satisfying $\omega_4=\omega_1+\omega_2-\omega_3$. As a result, waves with frequencies in the set $A_n$ will produce waves with frequencies in the set $A_{n+1}$. In order to consider a set that includes all possible frequencies, we cannot stop this process at any finite $n$ and it is therefore natural to define:
\begin{equation}\label{eq:a*}
 A^{\ast} =\bigcup_{n=1}^{\infty} A_n .
 \end{equation}
Notice that, if we start with a finite number of Dirac masses, namely $g^{in}=M_1 \delta_1 + \sum_{j=2}^{N_0}m_j \delta_j$ for some finite $N_0\geq 2$, it is easy to show that
\begin{equation}
	A_1=\{1,2,\dots,N_0\}\quad \Longrightarrow 
 %A_n=\{ 1,2,\ldots, 2^{n-1}+1\}, \qquad 
 \quad A^{\ast}=\N.
\end{equation}
We are now ready to state the following
\begin{lemma}[Lemma 3.5 in \cite{escobedo2015theory}] \label{thm:support1}
Let $\rho<-1$, $g^{in}\in \mathcal{M}_{\rho}$ and let $g$ be the weak solution to the WKE \eqref{eq:WKE}. Suppose that $M=\int_{\R_{+}} g^{in}(d\om) >0$. Then for any $x\in A^{\ast}$, any $t>0$ and any $r>0$ we have that 
\[ \int_{B_r (x)} g(t,d\om) >0.\]
\end{lemma}

The lemma above guarantees that $\N = A^{\ast} \subset \mbox{supp} (g(t))$ for all times $t>0$. In order to show the opposite inclusion, whence proving that $\mbox{supp} (g(t))=\N$, we need the following result.

\begin{lemma}[Lemma 3.8 in \cite{escobedo2015theory}] \label{thm:support2}
Let $\rho<-1$, $g^{in}\in \mathcal{M}_{\rho}$ and let $g$ be the weak solution to the WKE \eqref{eq:WKE}. Suppose that $M=\int_{\R_{+}} g^{in}(d\om) >0$ and that $\inf A^{\ast}>0$. Then $\mbox{supp}(g(t))\subset \overline{A^{\ast}}$ for any $t\geq 0$.
\end{lemma}

The dynamics of our problem are therefore discrete. Thanks to Lemmas \ref{thm:support1} and \ref{thm:support2}, we know that the functions
\begin{equation}
	F_n (t)= \int_{\R_{+}}\varphi_n(\omega)g(t,d\om)= \int_{\{n\}} g(t,d\om) \qquad n\geq 1,
\end{equation} 
as introduced in \eqref{def:Fn}, fully capture the dynamics of the problem. Thanks to the conservation of mass, we may assume that our solutions have unit mass. Then the conserved energy yields:
\begin{equation}
	\label{eq:consmass}
	\sum_{n=1}^{+\infty}F_n(t)=M_1+M_2=1, \qquad E=\sum_{n=1}^{\infty}nF_n(t)=M_1+E_2,
\end{equation}
where $M_2,E_2$ are as defined in Theorem \ref{th:main}.
In order to prove lower bounds on $F_2$, it is convenient to define
\begin{equation}
	\label{def:Hn}
	H_n(t):=(F_n(t))^{-1}, \qquad t>0.
\end{equation}
We know these functions are well defined, since Lemma \ref{thm:support1} guarantees that $F_n(t)>0$ for all $t>0$ and $n\in\N$. Next we state the (infinite) system of ODE's for these quantities.

\begin{lemma}\label{lem:tedious_equations}
	Let $F_n, H_n$ be respectively given in \eqref{def:Fn}, \eqref{def:Hn}. Then
	\begin{align}
		\label{eq:dtFn}
		\frac{d}{dt}F_n&=F_n (Q_n-U_n)-F_n^2L_n+C_n,\\
		\label{eq:dtHn}\frac{d}{dt}H_n&=H_n (U_n-Q_n)+L_n-H_n^2C_n
	\end{align}
where the terms $L_n,Q_n,U_n,C_n$ are defined as follows: 
\begin{equation}
	\label{def:Ln}
	L_n=\frac{2}{n}F_1+\frac{2}{n}\left(\sum_{k=2}^{n-1}F_k+\sum_{k=n+1}^{2n-1}\frac{F_k\sqrt{2n-k}}{\sqrt{k}}\right);
\end{equation}
\begin{align}
\label{def:Qn1}	Q_n=&\ \frac{4}{\sqrt{n}}\frac{F_1F_{2n-1}}{\sqrt{2n-1}}+\sum_{k=n+1}^{\infty}\frac{F_k^2}{k}+\frac{1}{\sqrt{n}}\sum_{k=\lceil \frac{n}{2}\rceil }^{n-1}\frac{F_k^2}{k}\sqrt{2k-n}\\
\notag &+\frac{2}{\sqrt{n}}\left(2\sum_{m=2}^{n-1}\frac{F_mF_{2n-m}}{\sqrt{2n-m}}+\sum_{k=\lceil \frac{n+1}{2}\rceil}^{n-1}\sum_{m=n+1-k}^{k-1}\frac{F_kF_m}{\sqrt{km}}\sqrt{k+m-n}\right);
\end{align}
\begin{align}
\label{def:Qn2}
		U_n=&\ \frac{4}{\sqrt{n}} F_1\sum_{k=2}^{n-1}\frac{F_k}{\sqrt{k}}+\frac{2}{\sqrt{n}}\left(\sum_{k=2}^{n-1}\sum_{m=n+1}^{n+k}\frac{F_kF_m}{\sqrt{km}}\sqrt{n+k-m}+2\sum_{k=2}^{n-1}\sum_{m=2}^{k-1}\frac{F_kF_m}{\sqrt{k}}\right)\\
		\notag &\ +\frac{2}{\sqrt{n}}\sum_{k=n+1}^{\infty}\sum_{m=k+1}^{n+k-1}\frac{F_kF_m}{\sqrt{km}}\sqrt{n+k-m};
\end{align}
\begin{equation}
	\label{def:Cn}
\begin{split}
	C_n=&\ 2F_1\sum_{k=n+1}^{\infty}\frac{F_kF_{k+m-n}}{\sqrt{k(k+m-n)}}+\sum_{k=\lceil \frac{n+1}{2}\rceil}^{n-1}\frac{F_k^2F_{2k-n}}{k}\\
 &\ +2\left(\sum_{k=2}^{n-1}\sum_{m=2}^{k-1}\frac{F_kF_mF_{k+m-n}}{\sqrt{kn}}+\sum_{k=n+1}^{\infty}\sum_{m=2}^{n-1}\frac{F_kF_mF_{k+m-n}}{\sqrt{k(k+m-n)}}\right)\\
&\  +\sqrt{n}\left(2\sum_{k=n+2}^{\infty}\sum_{m=n+1}^{k-1}\frac{F_kF_mF_{k+m-n}}{\sqrt{km(k+m-n)}}+\sum_{k=n+1}^{\infty}\frac{F_k^2F_{2k-n}}{k\sqrt{2k-n}}\right).	
\end{split}
\end{equation}
\end{lemma}
% \red{We should rewrite $C_n$ taking away $k=1$ or $m=1$, since I believe it they are never present. For instance, suppose that $m=1$ in the first term 
% \[
% \sum_{k=1}^{n-1}\sum_{m=1}^{k-1}\frac{F_kF_mF_{k+m-n}}{\sqrt{kn}}
% \]
% Then $k+m-n=k+1-n$ and since $k\leq n-1$ we have that $k+m-n\leq 0$, thus this term is non-existant. A similar reasoning rules out the possibility $k=1$. We should do this analysis for all the terms in \eqref{def:Cn} in order for our toy model to be clearer.}

The proof of \Cref{lem:tedious_equations} is a long albeit elementary computation, and so we postpone it to \Cref{sec:appendix}.
\begin{remark}
    In the definitions of $L_n,Q_n,U_n,C_n$, we have isolated the terms containing $F_1$ that appears as first terms on the right-hand side. For the long-term behavior, one should have in mind that $F_1=1$ up to small errors. Therefore when deriving a toy model or when doing a perturbative argument, it is natural to replace $F_1$ by $1$ and consider all the other terms as lower order terms to be neglected (more precisely, one would hope to bootstrap a suitable smallness condition).
\end{remark}

\section{Bounds on the rates of convergence}\label{sec:main}
In this section, we prove Theorem \ref{th:main} and Proposition \ref{prop:cond}. 
We  start by studying the function $F_1$. In particular, we want to be as quantitative as possible since, in view of the conservation of the mass \eqref{eq:consmass}, bounds on $F_1$ will yield a priori bounds for the rest of the $F_n$, $n\geq 2$.

	\begin{proposition}
		\label{prop:F1}
		Let $M_1,M_2\geq 0$ be such that $M_1+M_2=1$. Then for $t>t_0$ we have 
		\begin{align}\label{eq:decayF1prop}
			&1-F_1(t)\leq \frac{c_1(t_0)}{\sqrt{b_1(t_0)\,(t-t_0)+3E}},\\
		\label{eq:c1}
			&c_1(t_0)=\sqrt{3E}(1-F_1(t_0)), \qquad b_1(t_0)=2F_1(t_0)(1-F_1(t_0))^2.
		\end{align}
	\end{proposition}
When $t_0=0$, notice that $c_1(0)=:c_1=\sqrt{3E}(1-M_1)$ and $b_1(0)=:b_1=2M_1 (1-M_1)^2$.
%\begin{proposition}
%\label{prop:F1}
%Let $M_1>0$ and suppose that $M_1+m_2=1$. Then for $t>0$ we have 
%\begin{equation}\label{eq:decayF1prop}
%1-F_1(t)\leq \frac{c_1}{\sqrt{b_1\,t+3E}}
%\end{equation}
%where $c_1=\sqrt{3E}M_1=\sqrt{3(2-M_1)}\,M_1$ and $b_1=2m_1 (1-M_1)^2$.
%\end{proposition}
%\begin{remark}
%If $M_1=0$, we know by \Cref{thm:support1} that for any $t_0>0$ we have that $F_1(t_0)>0$. As a result, we can adapt \eqref{eq:decayF1prop} as follows: for any $t>t_0$ 
%\begin{equation}\label{eq:decayF1prop}
%1-F_1(t)\leq \frac{c_1(t_0)}{\sqrt{b_1(t_0)(t-t_0)+3E}}
%\end{equation}
%where the coefficients $c_1, b_1$ are given by
%\begin{equation}\label{eq:c1}
% c_1(t_0)=\sqrt{3E}(1-F_1(t_0)), \qquad b_1(t_0)=2F_1(t_0)(1-F_1(t_0))^2
%\end{equation}
%\end{remark}
Using \eqref{eq:decayF1prop}, we readily obtain the following.
\begin{corollary}
\label{cor:lowF} For $t>t_0$ we have that 
\begin{equation}
1-F_1(t)\geq  F_2 (t) \geq \frac{c_2}{(t-t_0+3E/b_1(t_0))^{\alpha}+C_2},
\end{equation}
where $c_2, C_2$ can be explicitly computed and
\begin{equation}
\label{def:alpha}
\alpha=\max\left\{1, \frac{3E}{16F_1(t_0)}\right\}. 
\end{equation}
%This implies 
%\begin{equation}\label{eq:decayF1}
%1-F_1(t)\geq \frac{c_2}{(t-t_0+3E/b_1(t_0))^{\alpha}+C_2} \qquad\mbox{for all}\ t>t_0.
%\end{equation}
%Moreover, if $M_1\geq 1/8$ then
%\[ \alpha= \max \left\lbrace 1, \left(\frac{2}{5\sqrt{2}}-\frac{1}{5}\right)\, c_1^2 \right\rbrace.\]
%In particular, $\alpha=1$ whenever $M_1\geq 0.9138$.
\end{corollary}
%\begin{remark}
%The exponent $\alpha$ can be estimated in terms of the initial datum. As an example, if $M_1\geq 1/2$, we have that $c_1<5.6$ and therefore $\alpha<2.6$. Whereas for $M_1$ close to 1, we have that $\alpha \rightarrow 0.841$.
%\end{remark}
\begin{remark}
Given that $F_1(t)\rightarrow 1$ as $t\rightarrow \infty$ monotonically, one can always choose $t_0$ large enough so that 
\begin{equation}\label{eq:cond_t0}
	\frac{3E}{16F_1(t_0)}\leq 1\quad \Longrightarrow \quad F_1(t_0)\geq \frac{3E}{16}.
\end{equation}
This is telling us that a lower bound with a rate $\mathcal{O}(t^{-\alpha})$ with $\alpha>1$ cannot be sustained for all times.
Moreover, since $E=E_2-M_1$, if we start with $M_1\geq 3E_2/19$ then we can take $t_0=0$ and $\alpha=1$. 
\end{remark}

The proof of Theorem \ref{th:main} directly follows by combining Proposition \ref{prop:F1} with Corollary \ref{cor:lowF}. Therefore, we just prove the latter results. The start of the proof of Proposition \ref{prop:F1} is similar to Theorem 3.2 in \cite{escobedo2015theory}. The convergence result presented in Theorem 3.2 in \cite{escobedo2015theory} is correct. However, the rate of convergence one could derive from the last differential inequality in the proof of Theorem 3.2 in \cite{escobedo2015theory} (see the second equality at page 53 in \cite{escobedo2015theory}) does not hold due to an error in said inequality. We fix this error as part of the proof of Proposition \ref{prop:F1}. In fact, if the convergence rate that can be deduced from \cite{escobedo2015theory} were true, we would be able to prove that $1-F_1(t)=\mathcal{O}((t-t_0)^{-1})$ and therefore Proposition \ref{prop:cond} would not be a conditional result.
 
\begin{proof}[\textbf{Proof of Proposition \ref{prop:F1}}]
  To obtain the bound for $F_1$, namely \eqref{bd:F2}, we follow the argument in the proof of \cite[Theorem 3.2]{escobedo2015theory}. In particular, choose the test function:
\begin{equation}
	\label{def:phi1}
	\varphi_{1}(\om):=\left(3-2\om\right)_+.
\end{equation}
Notice that $\mathrm{supp}(\varphi_{1})\subset[0,3/2]$ which implies $F_1(t)=\int \varphi_{1}(\om) g(t)$. 

Given that $\varphi_1$ in \eqref{def:phi1} is convex and that $g$ solves \eqref{eq:WKE}, we apply \Cref{thm:monotonicity} to get
\begin{equation}
\label{bd:F1'}
    F_1'(t)  = \frac{d}{dt} \left( \int_{\{1\}} g(t) \right) \geq \int_{\R_{+}^3} \frac{g_1g_2g_3}{\sqrt{\om_1\om_2\om_3}} \mathcal{G}_{\varphi_1},
\end{equation} 
where we omit the explicit dependencies on $t,d\om_j$ to ease the notation.
%\geq \frac{1}{3}\, \int_{\R_{+}^3} \frac{g_1g_2g_3}{\sqrt{\om_0\om_{+}}} \mathcal{H}_{\varphi_1}^{1}, 
Recalling the definitions of $\mathcal{G}_{\cdot}, \mathcal{H}_{\cdot}^{1}, \mathcal{H}_{\cdot}^2$ in \eqref{def:G}-\eqref{def:H2}, we claim that  
\begin{equation}
\label{bd:H2}
    \mathcal{H}_{\varphi_1}^2\geq0.
\end{equation}
Indeed, the coefficient of $\mathcal{H}_{\varphi_1}^2$ in the formula of $\mathcal{G}_{\varphi_1}$ is nonzero only if $\om_0+\om_{-}-\om_{+}>0$. Note also that $\mbox{supp}(\varphi_1) \cap \mbox{supp}(g(t)) = \{ 1\}$ so we only need to consider $\om_{+},\om_0,\om_{-}\in \N$, where the following happens:
\begin{itemize}
\item If $\om_{+}<1$ then $\om_0<1$ and $\om_{-}<1$ thus $\mathcal{H}_{\varphi_1}^2=0$.

\item If $\om_{+}=1$, we must have $\om_0=\om_{-}=1$ since $\om_0+\om_{-}-\om_{+}>0$. In this case $\mathcal{H}_{\varphi_1}^2=0$.

\item If $\om_{+}>1$, there are two options:
\begin{enumerate}
    \item Suppose at least one $\om_0$ or $\om_{-}$ is 1. Given that $\om_0+\om_{-}-\om_{+}>0$, the only option is $\om_{-}=1$ and $\om_0=\om_{+}$. In this case $\mathcal{H}_{\varphi_1}^2=0$.
    \item If $\om_0$ and $\om_{-}$ are not 1, then only $\varphi_1 (\om_0+\om_{-}-\om_{+})$ may be nonzero (if $\om_0+\om_{-}-\om_{+}=1$) and thus $\mathcal{H}_{\varphi_1}^2\geq 0$.
\end{enumerate}
\end{itemize}
Therefore, combining the inequality \eqref{bd:F1'} with \eqref{bd:H2} and the definition of $\mathcal{G}_{\cdot}$ \eqref{def:G}, we get
\begin{equation}
\label{bd:F1'H}
    F_1'(t)  \geq \frac{1}{3}\, \int_{\R_{+}^3} \frac{g_1g_2g_3}{\sqrt{\om_0\om_{+}}} \mathcal{H}_{\varphi_1}^{1}.
\end{equation} 
 We can also restrict integration to the set $\om_{+}>1$ or we would have $\mathcal{H}_{\varphi_1}^1=0$. From the definition of $\varphi_1$ we see that $\varphi_1 (\om_{+})=0$ as well as $\varphi_1 (\om_{+} + \om_0 - \om_{-})=0$. Consequently, 
\[  F_1'(t) \geq \frac{1}{3}\, \int_{\{\om_{+}>1, \, \om_j \in \N\}} \frac{g_1g_2g_3}{\sqrt{\om_0\om_{+}}} \, \varphi_1 (\om_{+} + \om_{-} - \om_0).\]
For $\varphi_1(\om_{+}+\om_{-}-\om_0)$ to be nonzero we must also have that $\om_0=\om_+>1$ and $\omega_{-}=1$. Indeed, if $\om_+-\om_0>0$, since all frequencies are concentrated in $\N$, we must have $\om_+-\om_0\geq 1$. Having also that $\omega_{-}\geq 1$, we conclude $\om_++\om_{-}-\om_0\geq 2$, which is outside the support of $\varphi_1$. Similarly, since $\omega_0\leq \omega_{+}$ if $\omega_{-}\geq 2$ we have $\omega_{+}+\omega_{-}-\omega_0\geq 2$.  Therefore
\begin{align}\label{eq:EV_error}
 F_1'(t) & 
 %\geq \frac{1}{3}\, \int_{\{\om_{+},\om_0>1, \, \om_j \in \N\}} \frac{g_1g_2g_3}{\sqrt{\om_0\om_{+}}} \, \varphi_1 (\om_{+} + \om_{-} - \om_0)
 \geq \frac{1}{3}\, \int_{\{\om_{-}=1,\ \om_0 = \om_{+}>1\}} \frac{g_1g_2g_3}{\sqrt{\om_0\om_{+}}}\nonumber \\
 & \geq  \frac{1}{3}\, F_1(t)\, \left( \int_{ \N -\{1\} } \frac{1}{\sqrt{\om}}\, g(t ) \right)^2 \geq \frac{1}{3}\,F_1(t_0)\, \left( \int_{ \N -\{1\} } \frac{1}{\sqrt{\om}}\, g(t) \right)^2.
 \end{align}
In the last inequality we used the fact that $F_1(t)$ is monotone nondecreasing. We may choose $t_0=0$ if $F_1(0)=M_1\neq 0$, whereas if $M_1=0$ any $t_0>0$ would do in view of \Cref{thm:support1}.
 
At this stage, we note that the factor $\omega^{-1/2}$ on the right-hand side of \eqref{eq:EV_error} was missing in the proof of Theorem 3.2 in \cite{escobedo2015theory},  and it is not clear why it can be removed. In fact, without it one  can simply exploit the conservation of mass to conclude, see \cite{escobedo2015theory}. Here we have to be more careful. We exploit both the conservation of the mass and energy and use an interpolation inequality. Namely, by the H\"older inequality
	\begin{align}
		 \int_{ \N -\{1\} } g(t )&\leq \bigg( \int_{ \N -\{1\} } \frac{1}{\sqrt{\om}}\, g(t ) \bigg)^{\frac23}\bigg( \int_{ \N -\{1\} }\om\, g(t ) \bigg)^{\frac13}\leq \bigg( \int_{ \N -\{1\} } \frac{1}{\sqrt{\om}}\, g(t ) \bigg)^{\frac23}E^\frac13,	\end{align}
where we used the conservation of the energy in the last inequality. Combining the bound above with \eqref{eq:EV_error}, and using the conserved, normalized mass, we find that 
\begin{align*}
	F_1'(t) & \geq\frac13 \frac{F_1(t_0)}{E}\left( \int_{ \N -\{1\} } \, g(t ) \right)^{3}=\frac{F_1(t_0)}{3E}\left( 1-F_1(t)\right)^{3}.
\end{align*}
Solving this differential inequality we obtain that
\begin{align}
	1-F_1(t)\leq \frac{c_1(t_0)}{\sqrt{b_1(t_0)(t-t_0)+3E}},
\end{align}
where $c_1,b_1$ are defined in \eqref{eq:c1}.
\end{proof}

\begin{proof}[\textbf{Proof of Corollary \ref{cor:lowF}}]
To prove the lower bound on $F_2$, it is convenient to make use of $H_2=F_2^{-1}$, which is well defined thanks to Lemma \ref{thm:support1}. On account of \eqref{eq:dtHn} we have that 
\begin{equation}\label{eq:diff_H2}
 \frac{d}{dt}H_2 \leq H_2  (U_2 - Q_2)+L_2,
\end{equation}
where
\begin{align}
	\label{eq:U2Q2}
&U_2 =  \frac{2}{\sqrt{2}}\sum_{k=3}^{\infty}\frac{F_kF_{k+1}}{\sqrt{k(k+1)}}, \qquad Q_2  = \frac{4}{\sqrt{6}}\, F_1 F_{3}+\sum_{k=3}^{\infty}\frac{F_k^2}{k},\\
\label{eq:L2}&L_2  = F_1 +\frac{F_3}{\sqrt{3}}\leq 2.
\end{align}
Using the Cauchy-Schwarz inequality, we bound $U_2$ as 
\begin{equation}
	U_2\leq \sum_{k=3}^\infty \frac{F_k^2}{k}+\frac12\sum_{k=4}^\infty \frac{F_k^2}{k}.
\end{equation}
For $k\geq 2$ we know that 
\begin{equation}
	F_k\leq  \sum_{j=2}^\infty F_j=1-F_1.
\end{equation}
Since $F_k>0$ for all $k$ (see \Cref{thm:support1}), \eqref{eq:decayF1prop} yields
\begin{equation}
	U_2-Q_2\leq \frac12\sum_{k=4}^\infty \frac{F_k^2}{k}\leq \frac{1}{8}(1-F_1)^2\leq \frac{1}{8}\frac{c_1(t_0)^2}{b_1(t_0)(t-t_0)+3E}
\end{equation}
Without loss of generality, let us assume that $t_0=0$ and $M_1>0$ from now on. We denote $c_1 (t_0)=c_1$ and $b_1(t_0)=b_1$ for this choice. Then
\begin{equation}
	\int_{s}^{t} (U_2 (\tau) - Q_2(\tau))\, d\tau \leq \int_{s}^{t}  \frac{1}{8}\frac{c_1^2}{b_1 \tau+3E} \, d\tau \leq  \frac{c_1^2}{8b_1} \log \left( \frac{b_1 t+3E}{b_1 s + 3E} \right),
\end{equation}
We define $\alpha=c_1^2/(8b_1)$ as announced in \eqref{def:alpha}. Then
\begin{equation}
	\label{bd:exp}
	 \exp\left( \int_{s}^{t} (U_2 (\tau) - Q_2(\tau))\, d\tau\right) \leq 
	 \left( \frac{b_1 t+3E}{b_1 s + 3E} \right)^{\alpha}\, .
\end{equation}
Integrating \eqref{eq:diff_H2} yields
\[ H_2(t)\leq e^{\int_{0}^t (U_2 (s) - Q_2(s)) \, ds} H_2(0)+\int_{0}^t e^{\int_{s}^t(U_2 (\tau) - Q_2(\tau))\, d\tau}L_2(s)\, ds.\]
If $\alpha\neq 1$, combining \eqref{bd:exp} with \eqref{eq:L2}, we estimate the second term as follows
\begin{align*}
	 \int_{0}^t e^{\int_{s}^t(U_2 (\tau) - Q_2(\tau))\, d\tau}L_2(s)\, ds  &\leq 2\,\int_{0}^{t} \left(\frac{t+3E/b_1}{s+3E/b_1}\right)^{\alpha} \, ds \\
	 &= \frac{2}{\alpha-1} \, \left[  (0+3E/b_1) \, \left(\frac{t+3E/b_1}{0+3E/b_1}\right)^{\alpha}-(t+3E/b_1)\right].
\end{align*}
 Putting these estimates together, we obtain
\[
H_2(t)\leq \frac{2}{\alpha-1} \, \left[  3E/b_1 \, \left(\frac{t+3E/b_1}{3E/b_1}\right)^{\alpha}-(t+3E/b_1)\right]
+\left(\frac{t+3E/b_1}{3E/b_1}\right)^{\alpha} \, H_2 (0).
\]
When $\alpha=1$ one would have a logarithmic correction instead of a power law in the bound above. Since $H_2(t)=F_2^{-1}(t)$, we deduce the following:  
\begin{itemize}
	\item When $\alpha> 1$, we conclude that $F_2 (t)\geq c_2((t+3E/b_1)^{\alpha}+C_2)^{-1}$.
	\item If $\alpha=1$ we have $F_2(t)\gtrsim c_2((t+3E/b_1)\log(t+3E/b_1)+C_2)^{-1}$
	\item When $\alpha<1$, the term in $t$ is dominant and therefore we have that $F_2 (t)\gtrsim c_2((t+3E/b_1)+C_2)^{-1}$.
\end{itemize}
This concludes the proof of the corollary.
%
%{\bf Case 2:} Next suppose that $M_1\geq 1/5$. Then we can improve over \eqref{eq:crossF2} as follows. First of all, we can rewrite 
%\[ U_2 = \frac{1}{\sqrt{6}} \, F_3 F_4 +  \sum_{k=4}^{\infty}\frac{F_kF_{k+1}}{\sqrt{k(k+1)}}.\]
%Note that the term in $F_1 F_3$ in $Q_2$, see \eqref{eq:U2Q2}, is larger than $\displaystyle \frac{1}{\sqrt{6}} \, F_3 F_4 $, i.e.
%\begin{align*}
%\frac{1}{\sqrt{6}} \, F_3 F_4 - \frac{4}{\sqrt{6}}\, F_1 F_{3} = \frac{F_3}{\sqrt{6}} \, \left( F_4- 4 F_1 \right)\leq 0
%\end{align*}
%at all times. This is because for $M_1\geq 1/5$ one has  
%$$4 F_1 (t) \geq 4 M_1\geq 1-M_1\geq F_4(t)$$ at all times. Therefore, we can improve on \eqref{eq:crossF2} as follows:
%\begin{align}\label{eq:crossF2_imp}
%\frac{2}{\sqrt{2}}\sum_{k=4}^{\infty}\frac{F_kF_{k+1}}{\sqrt{k(k+1)}}  \leq \frac{1}{\sqrt{2}} \sum_{k=4}^{\infty}\frac{F_k^2}{k} + \frac{1}{\sqrt{2}} \sum_{k=5}^{\infty} \frac{F_{k}^2}{k} = \frac{1}{4\sqrt{2}} F_4^2 + \frac{2}{\sqrt{2}} \sum_{k=5}^{\infty} \frac{F_{k}^2}{k}  
%\end{align}
%Notice that thanks to the first term in $Q_2$ we have that 
%\[ U_2 - Q_2 \leq \left( \frac{2}{\sqrt{2}} -1\right) \sum_{k=5}^{\infty} \frac{F_{k}^2}{k} \leq\left( \frac{2}{5\sqrt{2}} -\frac{1}{5}\right) \, \frac{c_1^2}{t} := \frac{\alpha}{t}. \] 
%This yields an improved $\alpha$ and concludes the proof of the corollary.
\end{proof}

Given the result in \Cref{th:main}, it is natural to ask what could be the maximal speed of convergence of $F_n$ with $n>1$.  If the lower bound for $1-F_1$ in \eqref{bd:F2} were sharp, then we can derive sharp lower and upper bounds for the speed of convergence of all the functions $F_n$ for $n$ large enough. This is the content of the following conditional result.

\begin{proposition}
	\label{prop:cond}
	Under the same assumptions as in \Cref{th:main}, suppose that the following inequality were true:
	\begin{equation}
		\label{bd:hypF1}
		1-F_1(t)\leq \frac{c}{t-t_0+C},
	\end{equation}
	for some constants $c,C>0$ and $t>t_0$. Then for any $n$ large enough such that 
	\begin{equation}
		\label{eq:gamman}
		\gamma_n:=\frac{2c}{\sqrt{n}}\left(\frac{1}{\sqrt{n+2}}+\mathbbm{1}_{\{n>2\}}\frac{1}{\sqrt{n+1}}+\mathbbm{1}_{\{n>2\}}\frac{2}{\sqrt{2}}\right)<1,
	\end{equation}  
	the following inequality holds for all $t>t_0$
	\begin{equation}
		\label{bd:lowFn}
		\frac{c_n}{t-t_0+C}\leq F_n(t)\leq \frac{c}{t-t_0+C},
	\end{equation}
	where $c_n$ can be explicitly computed. 
\end{proposition}
\begin{proof}
Recall the definitions of $L_n, U_n$ in \eqref{def:Ln}-\eqref{def:Qn2}. Thanks to the bound on the total mass, we deduce that $L_n\leq 4/n$. Therefore, from \eqref{eq:dtFn}  we  get
\begin{equation}
\label{bd:DtFn}
    \frac{d}{d t}F_n\geq-U_nF_n-\frac{4}{n}F_n^2.
\end{equation}
Regarding $U_n$, exploiting the conservation of the mass, we have 
\begin{equation}
	\label{bd:Un1}
	\sum_{k=n+1}^{\infty}\sum_{m=k+1}^{n+k-1}\frac{F_kF_m}{\sqrt{km}}\sqrt{n+k-m}\leq\frac{1}{\sqrt{n+2}}	\sum_{k=n+1}^{\infty}F_k\sum_{m=n+2}^{\infty}F_m\leq\frac{1}{\sqrt{n+2}}(1-F_1)^2. 
\end{equation}
If $n>2$, we have  additional terms in $U_n$, which we bound as follows:
\begin{equation}
	\label{bd:Un2}
	\sum_{k=2}^{n-1}\sum_{m=n+1}^{n+k}\frac{F_kF_m}{\sqrt{km}}\sqrt{n+k-m}\leq\frac{1}{\sqrt{n+1}}	\sum_{k=2}^{n-1}F_k\sum_{m=n+1}^{n+k}F_m\leq\frac{1}{\sqrt{n+1}}(1-F_1)^2,
\end{equation}
and 
\begin{equation}
	\label{bd:Un3}
	2\sum_{k=2}^{n-1}\sum_{m=1}^{k-1}\frac{F_k F_m}{\sqrt{k}}\leq\frac{2}{\sqrt{2}}\left(F_1\sum_{k=2}^{n-1}F_k+\sum_{k=2}^{n-1}F_k \sum_{m=2}^{k-1}F_m\right)\leq \frac{2}{\sqrt{2}}(1-F_1)(F_1+(1-F_1)).
\end{equation}
The upper bound above is the worst in terms of decay in time since it contains a factor $F_1$. Therefore, since $(1-F_1)\leq1$, combining \eqref{bd:Un1}, \eqref{bd:Un2} and \eqref{bd:Un3} we obtain
\begin{equation}
	\label{bd:Un}
	U_n\leq \tilde{\gamma}_n(1-F_1), \qquad \tilde{\gamma}_n=\frac{2}{\sqrt{n}}\left(\frac{1}{\sqrt{n+2}}+\chi_{n>2}\frac{1}{\sqrt{n+1}}+\chi_{n>2}\frac{2}{\sqrt{2}}\right).
\end{equation} 
For simplicity of notation, consider now $t_0=0$ in \eqref{bd:hypF1}. Combining \eqref{bd:DtFn} with \eqref{bd:Un} we obtain 
\begin{equation}
	\frac{d}{d t}F_n\geq  -\frac{\gamma_n}{t+C}F_n-\frac{4}{n}F_n^2,
\end{equation}
where $\gamma_n=\tilde{\gamma}_nc$ was given in \eqref{eq:gamman}. Defining $G_n:=(t+C)^{\gamma_n}F_n$, we find
\begin{equation}
	\frac{d}{dt}G_n\geq-\frac{4}{n(t+C)^{\gamma_n}}G_n^2.
\end{equation}
Since $\gamma_n<1$ by hypothesis, by a comparison principle we get
\begin{equation}
	\frac{1}{G_n (0)} - \frac{1}{G_n (t)}  \geq -\frac{4}{n (1-\gamma_n)} \, \left((t+C)^{1-\gamma_n}-C^{1-\gamma_n}\right),
\end{equation}
which immediately implies 
	\begin{equation}
		\label{bd:prooflowFn}
		F_n(t)\geq \frac{1}{ \frac{4}{n(1-\gamma_n)} \, \left((t+C) -C^{1-\gamma_n} (t+C)^{\gamma_n}\right)\, +(CF_n (0))^{-1}(t+C)^{\gamma_n} }, 
	\end{equation}
whence proving Proposition \ref{prop:cond}, where $c_n$ can be computed from the inequality above.
\end{proof}

\section{Toy model}\label{sec:toy_model}

In this section we discuss the toy model announced in the introduction \eqref{eq:intro_toy_model}. Our main goal is to understand if there could be solutions exhibiting a decay of order $\mathcal{O}(t^{-1})$, which could be possible for particular initial data and would justify the optimality of the lower bound in Theorem \ref{th:main}. Investigating this question directly on the full system \eqref{eq:dtFn} seems hard. For example, a naive ansatz imposing a polynomial decay for $F_n$ is not consistent with the equations, and this is because $F_1$ cannot decay and it behaves different with respect to all the other $F_n$. Therefore, we first aim at reducing the complexity of the system by performing the following reductions:
\begin{itemize}
\item  By \eqref{eq:dtFn}, for $n\geq 2$, $dF_n/dt$ is a weighted sum of products of the form $F_i F_j F_k$. We drop all terms where $\{i,j,k\}\cap \{1\}=\emptyset$.
\item All the remaining terms have at most one factor of $F_1$. In view of \Cref{prop:F1}, we substitute such terms by $1$. 
\end{itemize}
The resulting toy model may be written as
\begin{equation}\label{eq:real_toy_model}
\frac{d}{dt}F_n = 4\,  \frac{F_n}{\sqrt{n}} \, \frac{F_{2n-1}}{\sqrt{2n-1}} - 4\, \frac{F_n}{\sqrt{n}}\, \sum_{k=2}^{n-1} \frac{F_k}{\sqrt{k}}- 2 \left(\frac{F_n}{\sqrt{n}}\right)^2+ 2\sum_{k=n+1}^{\infty} \frac{F_kF_{k+1-n}}{\sqrt{k(k+1-n)}}.
\end{equation}
    The idea behind our approximating toy model is that the leading order terms are dictated by interactions with $F_1$. Indeed, we know that all the mass is converging towards $F_1$ and therefore interactions between $F_j$ with $j\neq 1$ are lower order. In this toy model though, many nice properties such as positivity of the solution cannot be expected to hold anymore.
 On the other hand, the advantage of the equation \eqref{eq:real_toy_model} is that the right-hand side is quadratic. We thus propose the natural self-similar \emph{ansatz} of the form
\[
F_n(t) = \beta_n\frac{\sqrt{n}}{t}, \qquad \text{for } 2\leq n\in \N.
\]
We plug this ansatz into \eqref{eq:real_toy_model} and derive equations for the coefficients $(\beta_n)_{n\geq 2}$
\begin{equation}\label{eq:real_beta}
-\sqrt{n}\, \beta_n = 4\, \beta_n \, \beta_{2n-1}  - 4\, \beta_n \, \sum_{k=2}^{n-1} \beta_k- 2 \beta_n^2+ 2\sum_{k=n+1}^{\infty} \beta_k \beta_{k+1-n}.
\end{equation}
Our goal is to investigate wheter or not there are positive solutions to this toy model, always with the idea in mind of having something consistent with the behavior observed in the full WKE, especially regarding \Cref{thm:support1}-\Cref{thm:support2} about the positivity of the coefficients $\beta_n$. We do not focus too much on the mass and energy properties since one can suitably rescale the time in the self-similar ansatz to adjust the parameters.

For the sake of understanding this toy model, we would like to introduce a suitable truncation and exhibit positive solutions to  truncated system. The aim would be to use such solutions as the starting point of a perturbative argument in the full WKE. First of all, we observe the following:
\begin{remark}
    Consider a solution of the form $\beta_n=0$ for all $n\geq N$. In the case $N=4$, it is straightforward to check that $\beta_2,\beta_3$ must be given by 
\[
\beta_2 = \frac{\sqrt{2}+3\sqrt{3}}{14}>0, \qquad \beta_3=\frac{-2\sqrt{2}+\sqrt{3}}{14}<0,
\]
after imposing $\beta_2,\beta_3\neq 0$. Similarly, if $N=5$ the only real-valued solutions have $\beta_4<0$. For $N=6$, one can numerically compute the four exact (nonzero) real-valued solutions to the system, but none of them lies in $(0,\infty)^4$.
    
These examples suggest that if we truncate brutally all the $n\geq N$, there is no guarantee of finding strictly positive solutions to our system, which is clearly not consistent with \Cref{thm:support1}.
\end{remark}

\subsection{Truncated system}
We consider the truncated system for $\bm{\beta}:=(\beta_2,\ldots,\beta_{N})$ by setting
\begin{equation}\label{eq:beta_zeroes}
\beta_k=0 \quad \mbox{for all}\ N+1 \leq k < 2N,\ \mbox{and}\ k\geq 3N-1.
\end{equation}
We keep $N-1$ functions $\bm{\lambda}:=(\beta_{2N},\beta_{2N+1},\ldots, \beta_{3N-2})$ as \emph{parameters} which are a priori fixed. 
With this in mind, \eqref{eq:real_beta} reads
\begin{equation}\label{eq:beta_truncated_0}
-\sqrt{n}\, \beta_n = 4\, \beta_n \, \beta_{2n-1} - 4\, \beta_n \, \sum_{k=2}^{n-1}\beta_k - 2 \beta_n^2 + 2\sum_{k=n+1}^{3N-2} \beta_k \beta_{k+1-n} \qquad \mbox{for}\ n=2,\ldots, N.
\end{equation}
Exploiting \eqref{eq:beta_zeroes}, we may rewrite this as:
\begin{equation}\label{eq:beta_truncated}
\begin{split}
-\sqrt{n}\, \beta_n = &\ 4\, \beta_n \, \beta_{2n-1}  \mathbbm{1}_{\{2n-1\leq N\}} - 4\, \beta_n \, \sum_{k=2}^{n-1} \beta_k- 2 \beta_n^2\\
& + 2\sum_{k=n+1}^{N} \beta_k \beta_{k+1-n} +  2\sum_{k=2N+1}^{3N-2} \beta_k \beta_{k+1-n}  \qquad \mbox{for}\ n=2,\ldots, N.
\end{split}
\end{equation}
For this system, we have a special solution given by

\[
\begin{split}
\label{def:particular_sol}\bm{\beta}^0  = ( \gamma, 0, \ldots ,0), \qquad 
\bm{\lambda}^0  = (\lambda_1^0 , \lambda_2^0, 0, \ldots,0),
\end{split}
\]
where
\begin{equation}
    \label{def:gamma_toy}
\gamma = \frac{\sqrt{2} + \sqrt{2+ 16\lambda_1^{0}\lambda_2^{0} }}{4}.
\end{equation}
Around this particular solution, we are able to show the existence of many solutions $\bm{\beta}$  to \eqref{eq:beta_truncated} with strictly positive components. Therefore,  we have many solutions for the system \eqref{eq:real_beta} when truncated at $N$ except for the coefficients $2N,2N+1$. This is the content of Theorem \ref{thm:toy_model} which we restate more precisely as follows
\begin{theorem}\label{thm:nontrivial_solution_inhom2} Fix any $N\in\N$ with $N\geq 4$. Fix $\lambda_1^0, \lambda_2^0>0$ such that 
\begin{equation}
    \lambda^0_1\lambda^0_2 > N/4.
\end{equation}
Then there exists some $\delta_0 (N,\lambda_1^0, \lambda_2^0)>0$ such that for all $\delta<\delta_0$, there exists a solution to \eqref{eq:beta_truncated} (which solves \eqref{eq:real_beta} for $n\leq N$) such that $\beta_2,\ldots, \beta_N,\beta_{2N}, \beta_{3N-2}>0$, and $\beta_k=0$ for the rest of $k\in\N\cap [2,\infty)$. Moreover, this solution satisfies
\begin{equation}
    \begin{split}
        \beta_2 & = \gamma + \O ( \sqrt{N} \,\delta\, \max\{ \lambda^0_1,\lambda^0_2\}),\\
        \beta_j & =  \O ( \sqrt{N} \,\delta\, \max\{ \lambda^0_1,\lambda^0_2\}) \qquad j=3,\ldots,N,\\
        \beta_{2N} & = \lambda^0_1 + \O (\delta),\\
         \beta_{2N+1} & = \lambda^0_2 + \O (\delta),\\
         \beta_{j} & = \O(\delta), \qquad j=2N+2,\ldots, 3N-2.
    \end{split}
\end{equation}
\end{theorem}
\begin{proof}
The idea of the proof is to construct the positive solutions by using the implicit function theorem for a map whose zeros are solutions to \eqref{eq:beta_truncated_0}. We have to carefully set the parameters in the special solution \eqref{def:particular_sol} in order to guarantee the positivity of the new solution. The proof is divided into four steps.

\medskip 
\noindent
\textbf{Step 1.}
Consider the map
\[
\begin{split}
f: \R^{N-1}\times & \R^{N-1}   \longrightarrow \R^{N-1} \\
f(\bm{\beta},\bm{\lambda}) & = (f_n (\bm{\beta},\bm{\lambda}) )_{n=2}^{N}
\end{split}
\]
where
\begin{equation}
\begin{split}
f_n (\bm{\beta},\bm{\lambda}) =& \ 2 \beta_n^2 + 4\, \beta_n \, \sum_{k=2}^{n-1} \beta_k -4\, \beta_n \, \beta_{2n-1}  \mathbbm{1}_{\{2n-1\leq N\}}\\
& - 2\sum_{k=n+1}^{N} \beta_k \beta_{k+1-n} -2  \sum_{k=2N+1}^{3N-2} \beta_k \beta_{k+1-n} -  \sqrt{n} \beta_n
\end{split}
\end{equation}
 and 
\[
\begin{split}
\bm{\beta} & = ( \beta_2,\ldots,\beta_N)\\
\bm{\lambda} & = (\lambda_1 , \ldots, \lambda_{N-1})= ( \beta_{2N},\ldots,\beta_{3N-2}).
\end{split}
\]
Notice that $f_n(\bm{\beta},\bm{\lambda})=0$ for all $2\leq n\leq N$ corresponds to solution  to \eqref{eq:beta_truncated_0}. We thus consider the special point:
\[
\begin{split}
\bm{\beta}^{0} & = ( \gamma, 0, \ldots ,0)\\
\bm{\lambda}^{0} & = (\lambda_1^{0} , \lambda_2^{0}, 0, \ldots,0),
\end{split}
\]
where $\gamma$ is defined in \eqref{def:gamma_toy}.
We know that $f (\bm{\beta}^{0},\bm{\lambda}^{0} )=0$. Moreover, the Jacobian matrix at this point is non-singular. More precisely, 
\begin{equation}\label{eq:jacobian2}
J_{\beta} f (\bm{\beta}^{0},\bm{\lambda}^{0} ) = \mbox{diag}\left( 4\gamma -  \sqrt{n}\right)_{2\leq n\leq N} + \left( -2 \gamma\, \delta_{(i,j)=(i,i+1)}\right)_{1\leq i,j\leq N-1}.
\end{equation}
Notice that this  is an upper triangular  matrix, meaning that it is invertible provided
\[
4\gamma -  \sqrt{n} \neq 0 \qquad \mbox{for all}\ n=2,\ldots, N-1.
\]
For reasons that will be clear later, we impose that 
\begin{equation}\label{eq:cond_jacobian}
\gamma >  \frac14\sqrt{N},
\end{equation}
which implies that every diagonal entry in \eqref{eq:jacobian2} is strictly positive.

By the Implicit Function Theorem, there exist $\varepsilon,\delta>0$ such that $\bm{\beta}$ can be written as a smooth function of $\bm{\lambda}$ in small neighborhoods of our zero, i.e.
\[
\begin{split}
\bm{\beta}:  B(\bm{\lambda}^{0}, \delta) & \longrightarrow B(\bm{\beta}^{0},\varepsilon)\\
\bm{\lambda} & \longmapsto \bm{\beta}(\bm{\lambda}).
\end{split}
\]
and such that $f(\bm{\beta}(\bm{\lambda}),\bm{\lambda})=0$ for all $\lambda \in B(\bm{\lambda}^{0}, \delta)$. Moreover, we know that 
\begin{equation}\label{eq:jacobian_beta3}
J_{\bm{\lambda}} \bm{\beta} (\bm{\lambda}^{0}) = - J_{\bm{\beta}}f (\bm{\beta}^{0},\bm{\lambda}^{0})^{-1} \, \cdot\,  J_{\bm{\lambda}}f (\bm{\beta}^{0},\bm{\lambda}^{0}).
\end{equation}
\medskip

\textbf{Step 2.} We now compute the Jacobian matrix $J_{\bm{\lambda}} \bm{\beta} (\bm{\lambda}^{0})$. Set 
\begin{equation}\label{eq:bn}
c_n=\frac{2\gamma}{4\gamma-\sqrt{n}} \qquad n=2,\ldots, N-1,
\end{equation}
and let 
\[
A:= \begin{pmatrix}
1 & c_2 & c_2 c_3 & c_2 c_3 c_4 & \ldots & \prod_{j=2}^{N-1} c_j \\
0 & 1 & c_3 & c_3 c_4 & \ldots & \prod_{j=3}^{n-1} c_j\\
0 & 0 & 1 & c_4 & \ldots & \prod_{j=4}^{N-1} c_j\\
\vdots & & \ddots & \ddots & \ddots & \vdots\\
\vdots & & & \ddots & \ddots & c_{N-1}\\
0 & & \ldots & & 0 & 1
\end{pmatrix}
\]
Then, it is easy to check that 
\begin{equation}
\label{eq:Jbeta-1}
J_{\bm{\beta}}f (\bm{\beta}^{0},\bm{\lambda}^{0})^{-1}= A \cdot \mbox{diag}\left( \frac{1}{4\gamma - \sqrt{n}}\right)_{2\leq n \leq N},
\end{equation}
which is an upper triangular matrix with strictly positive entries in the upper triangle, thanks to \eqref{eq:cond_jacobian} and \eqref{eq:bn}.
Similarly, we may compute
\begin{equation}\label{eq:jacobian_lambda2}
J_{\bm{\lambda}}f (\bm{\beta}^{0},\bm{\lambda}^{0}) =
- \begin{pmatrix}
\lambda_1^{0} & \lambda_2^{0} & 0 & 0 &0 & \ldots & 0 \\
0 & \lambda_1^{0} & \lambda_2^{0} & 0 & 0 &\ldots & 0 \\
0 & 0 & \lambda_1^{0} & \lambda_2^{0} & 0 & \ldots & 0 \\
\vdots & & \ddots & \ddots &  \ddots & & \vdots \\
\vdots & & & \ddots & \ddots &  \ddots & \vdots \\
\vdots & & & & \ddots & \ddots &   \\
0 & \cdots & &   \cdots & &  &  \lambda_1^{0}
\end{pmatrix}.
\end{equation}
Therefore, using \eqref{eq:jacobian_beta3} it is not hard to deduce that  $J_{\bm{\lambda}} \bm{\beta} (\bm{\lambda}^{0})$ is an upper triangular matrix with strictly positive entries on 
the upper triangle.

\medskip

\textbf{Step 3.} We are finally in the position of constructing the solution $\bm{\beta}$ with all positive entries. By the Fundamental Theorem of Calculus, we have that 
\begin{equation}\label{eq:beta_lambda}
\bm{\beta} ( \bm{\lambda}) = \bm{\beta}^{0} + \int_0^1 J_{\bm{\lambda}} \bm{\beta} ( \bm{\lambda}^{0} + t\, (\bm{\lambda}-\bm{\lambda}^{0}))\cdot (\bm{\lambda}-\bm{\lambda}^{0})\, dt 
\end{equation}
where $J_{\bm{\lambda}} \bm{\beta}$ is the Jacobian matrix of $\bm{\beta}$ with respect to $\bm{\lambda}$. 

Let us choose $\bm{\lambda}\in B(\bm{\lambda}^{0}, \delta)$ such that $\bm{\lambda}- \bm{\lambda}^{0}\in (0,\infty)^{N-1}$. Since the entries of the matrix $J_{\bm{\lambda}} \bm{\beta} (\bm{\lambda}^{0})$ are  strictly positive in the upper triangle, we have that 
\[
[J_{\bm{\lambda}} \bm{\beta} ( \bm{\lambda}^{0} + t\, (\bm{\lambda}-\bm{\lambda}^{0}))\cdot (\bm{\lambda}-\bm{\lambda}^{0})]_j \Big |_{t=0} > 0 \qquad \mbox{for all} \quad j=1,\ldots, N-1.
\]
By continuity of the Jacobian matrix, we can extend this positivity to any $t\in [0,1]$ as long as $\bm{\lambda}-\bm{\lambda}^{0}$ is small enough (i.e. by potentially making $\delta>0$ smaller). By \eqref{eq:beta_lambda}, this implies that 
\[
[\bm{\beta} ( \bm{\lambda}) ]_j > 0 \qquad \mbox{for all} \quad j=1,\ldots, N-1.
\]
\medskip

\textbf{Step 4.} 
Let us further impose
\begin{equation}\label{eq:cond_lambda}
    \lambda_1^0 \lambda_2^0>\frac14 N
\end{equation}
in order to guarantee that $\gamma>\sqrt{N}/2$. This implies that $c_n$ in \eqref{eq:bn} satisfies $c_n\leq 1$ and therefore the entries of the matrix in \eqref{eq:Jbeta-1} are of size $1/\sqrt{N}$. Hence, the entries of the matrix $J_{\bm{\lambda}} \bm{\beta} (\bm{\lambda}^{0})$, see  \eqref{eq:jacobian_beta3}, have size $\|{\boldsymbol{\lambda}^{0}}\|_{\infty}/\sqrt{N}$. Therefore, by summing up at most $N$-terms of size $\|{\boldsymbol{\lambda}^{0}}\|_{\infty}/\sqrt{N}$, we infer 
\[
\norm{J_{\bm{\lambda}} \bm{\beta} ( \bm{\lambda}^{0} ))\cdot (\bm{\lambda}-\bm{\lambda}^{0})}_{\infty}\lesssim \sqrt{N} \, \max\{\lambda_1^{0},\lambda_2^{0}\}\, \delta.
\]
By the continuity of the Jacobian matrix, one can arrange:
\[
\norm{J_{\bm{\lambda}} \bm{\beta} ( \bm{\lambda}^{0} + t\, (\bm{\lambda}-\bm{\lambda}^{0})) ))\cdot (\bm{\lambda}-\bm{\lambda}^{0})}_{\infty}\lesssim \sqrt{N} \, \max\{\lambda_1^{0},\lambda_2^{0}\}\, \delta \qquad \forall t\in [0,1],
\]
by further reducing $\delta$ if necessary. Therefore, in view of \eqref{eq:beta_lambda}, we choose $\delta$ small enough so that 
\begin{equation}
\begin{split}
\bm{\beta}(\bm{\lambda})_1  & =\beta_2 (\bm{\lambda})= \gamma + \O ( \sqrt{N}\norm{\bm{\lambda}^0}_{\infty}\, \delta),\\
\bm{\beta} (\bm{\lambda})_j & = \O ( \sqrt{N}\norm{\bm{\lambda}^0}_{\infty}\, \delta),\qquad j=2,\ldots,N-1\\
\bm{\lambda}_1 & = \beta_{2N}=\lambda_1^0 + \O (\delta)\\
\bm{\lambda}_2 & = \beta_{2N+1}=\lambda_2^0 + \O (\delta)\\
\bm{\lambda}_j & =\beta_{2N+j-1} = \O (\delta) \qquad j=3,\ldots, N-1,
\end{split}
\end{equation}
where $\gamma$ is defined in \eqref{def:gamma_toy} and we impose \eqref{eq:cond_lambda}. This concludes the proof.
\end{proof}

\section{Derivation of the discrete system}\label{sec:appendix}
	
We are going to derive the equations of $F_n$ by the weak formulation \eqref{eq:WKE}. As test functions, we choose
\begin{equation}
	\label{def:phin}
	\varphi_n(\om)= \chi_{(n-1/2,n+1/2)}, \qquad \text{for }  n\in \mathbb{N}\setminus \{0\}.
\end{equation}
where the $\chi$ are $C_c^{\infty}(\R)$ functions supported inside intervals of the form $(n-1/2,n+1/2)$ and such that $\varphi_n (n)=1$.

	From the definition of $F_n$, see \eqref{def:Fn}, and \eqref{eq:WKE} we have
	\begin{equation}
		\label{eq:detFn}
\begin{split}
			\de_t F_n&=\int_{\R^3_+}\Phi \frac{g_1g_2g_3}{\sqrt{\omega_1\omega_2\omega_3}}[\varphi_{n,4}+\varphi_{n,3}-\varphi_{n,1}-\varphi_{n,2}]d\om_1d\om_2d\om_3=:\mathcal{I}[g,\varphi_n],\\
			\om_4&=\om_1+\om_2-\om_3.
\end{split}
	\end{equation}	
	Notice that we always have $\om_3\neq \om_2$ since otherwise also $\om_4=\om_1$  and the integrand above vanishes. Analogously, we have $\om_4\neq \om_1$.
	
	We have to distinguish several cases depending on the values of $\varphi_{n,1}$.
	
	 \medskip $\bullet$ \textbf{Case} $\varphi_{n,1}=\varphi_{n,2}=1$. 
	First notice that we have $\varphi_{n,3}=\varphi_{n,4}=0$. Indeed, if $\varphi_{n,3}=1$, this implies $\omega_4=\omega_1+\omega_2-\om_3=n$, meaning that $\varphi_{n,4}=1$. But if $\varphi_{n,i}=1$ for $i=1,\dots,4$ the integrand is zero. Therefore we have $\varphi_{n,3}=\varphi_{n,4}=0$. 
	
	When $\om_3<n$ then $\Phi=\sqrt{\om_3}$, hence
	\begin{equation}
		\label{eq:11213<n}
		\mathcal{I}[g,\varphi_n\mathbbm{1}_{\{\varphi_{n,1}=\varphi_{n,2}=1\}\cap \{\om_3<n\}}]=-\frac{2}{n}F_n^2\sum_{k=1}^{n-1}F_k.
	\end{equation}
	For $\om_3>n$ we have $\Phi=\sqrt{\om_4}=\sqrt{2n-\om_3}$, therefore
	\begin{equation}
		\label{eq:11213>n}
		\mathcal{I}[g,\varphi_n\mathbbm{1}_{\{\varphi_{n,1}=\varphi_{n,2}=1\}\cap \{\om_3>n\}}]=-\frac{2}{n}F_n^2\sum_{k=n+1}^{2n-1}\frac{F_k\sqrt{2n-k}}{\sqrt{k}}.
	\end{equation}

	\medskip $\bullet$ \textbf{Cases $\varphi_{n,1}=1, \ \varphi_{n,2}=0$ or $\varphi_{n,1}=0, \ \varphi_{n,2}=1$}. 
	In view of the symmetry of the integrals, the two cases under consideration are equal. If $\varphi_{n,3}=1$, then $\om_4=\om_2$ meaning that $\varphi_{n,4}=0$. But if $\varphi_{n,1}=\varphi_{n,3}=1$ and $\varphi_{n,2}=\varphi_{n,4}=0$ then the integrand is zero. Analogously when $\varphi_{n,4}=0$. Hence we only have to consider $\varphi_{n,3}=\varphi_{n,4}=0$. 
	
	We use the following convention for the indices in this case $$k=\om_2 \text{ and } m=\om_3.$$ We start with the case $\om_2<n$. When $\om_3>n$ then $\Phi=\sqrt{n+k-m}$ and 
	\begin{equation}
		\label{eq:1120<n3>n}
		\mathcal{I}[g,\varphi_n\mathbbm{1}_{\{\varphi_{n,1}=1, \varphi_{n,2}=0\}\cap \{\om_2<n,\om_3>n\}}]=-\frac{1}{\sqrt{n}}F_n\sum_{k=1}^{n-1}\sum_{m=n+1}^{n+k}\frac{F_kF_m}{\sqrt{km}}\sqrt{n+k-m}.
	\end{equation}
	When $\om_3<\om_2$ one has $\Phi=\sqrt{m}$ meaning that
	\begin{equation}
		\label{eq:1120<n3<2}
		\mathcal{I}[g,\varphi_n\mathbbm{1}_{\{\varphi_{n,1}=1, \varphi_{n,2}=0\}\cap \{\om_3<\om_2<n\}}]=-\frac{1}{\sqrt{n}}F_n\sum_{k=2}^{n-1}\sum_{m=1}^{k-1}\frac{F_kF_m}{\sqrt{k}}.
	\end{equation}
	If $\om_2<\om_3<n$ then $\Phi=\sqrt{k}$ so that
	\begin{equation}
		\label{eq:1120<n3>2}
		\mathcal{I}[g,\varphi_n\mathbbm{1}_{\{\varphi_{n,1}=1, \varphi_{n,2}=0\}\cap \{\om_2<\om_3<n\}}]=-\frac{1}{\sqrt{n}}F_n\sum_{m=2}^{n-1}\sum_{k=1}^{m-1}\frac{F_kF_m}{\sqrt{m}}.
	\end{equation}
	When $\om_2>n$, first consider $\om_3<\om_2$. If $\om_3<n$ then $\Phi=\sqrt{m}$ and
	\begin{equation}
		\label{eq:1120>n3<n}
		\mathcal{I}[g,\varphi_n\mathbbm{1}_{\{\varphi_{n,1}=1, \varphi_{n,2}=0\}\cap \{\om_3<n<\om_2\}}]=-\frac{1}{\sqrt{n}}F_n\sum_{k=n+1}^{\infty}\sum_{m=1}^{n-1}\frac{F_kF_m}{\sqrt{k}}.
	\end{equation}
	For $n<\om_3<\om_2$ one has $\Phi=\sqrt{n}$ hence we get
	\begin{equation}
		\label{eq:1120>3>n}
		\mathcal{I}[g,\varphi_n\mathbbm{1}_{\{\varphi_{n,1}=1, \varphi_{n,2}=0\}\cap \{n<\om_3<\om_2\}}]=-F_n\sum_{k=n+2}^{\infty}\sum_{m=n+1}^{k-1}\frac{F_kF_m}{\sqrt{km}}.
	\end{equation}
	For $\om_3>\om_2>n$ then $\Phi=\sqrt{n+k-m}$ and
	\begin{equation}
		\label{eq:1120<3>n}
		\mathcal{I}[g,\varphi_n\mathbbm{1}_{\{\varphi_{n,1}=1, \varphi_{n,2}=0\}\cap \{n<\om_2<\om_3\}}]=-\frac{1}{\sqrt{n}}F_n\sum_{k=n+1}^{\infty}\sum_{m=k+1}^{n+k-1}\frac{F_kF_m}{\sqrt{km}}\sqrt{n+k-m}.
	\end{equation}
	This concludes all the possibles cases for $\varphi_{n,1}=1, \varphi_{n,2}=0$. In account of the symmetry $\om_1\leftrightarrow\om_2$, we remark again that all the terms appearing here are multiplied by a factor $2$ in  \eqref{eq:dtFn}.
	
	\medskip $\bullet$ \textbf{Case $\varphi_{n,3}=\varphi_{n,4}=1$}
	In this case we also have $\varphi_{n,1}=\varphi_{n,2}=0$ since otherwise the integrand is zero. The indices convention in this case are 
	$$k=\om_2 \text{ and } m=\om_1.$$
	We also have $\om_1+\om_2=2n$, hence, if $\om_1<\om_2$ then $\om_2>n$ and $\Phi=\sqrt{m}$. Since we can always exchange $\om_1$ and $\om_2$ we conclude that 
	\begin{equation}
		\label{eq:3141}
		\mathcal{I}[g,\varphi_n\mathbbm{1}_{\{\varphi_{n,3}=\varphi_{n,4}=1\}}]=\frac{4}{\sqrt{n}}F_n\sum_{m=1}^{n-1}\frac{F_mF_{2n-m}}{\sqrt{2n-m}}.
	\end{equation}

	\medskip $\bullet$ \textbf{Case $\varphi_{n,3}=1,\varphi_{n,4}=0$}
	In this case we know that $\varphi_{n,1}=\varphi_{n,2}=0$ since otherwise the integrand is zero. We again denote $$k=\om_2 \text{ and } m=\om_1.$$
	First we consider $\om_1<\om_2$. In account of the symmetries, the case $\om_2<\om_1$ will be the same. If $\om_2<n$ then $\Phi=\sqrt{k+m-n}$ and 
	\begin{equation}
		\label{eq:31401<2<n}
		\mathcal{I}[g,\varphi_n\mathbbm{1}_{\{\varphi_{n,3}=1, \varphi_{n,4}=0\}\cap \{\om_1<\om_2<n\}}]=\frac{1}{\sqrt{n}}F_n\sum_{k=\lceil \frac{n+1}{2}\rceil}^{n-1}\sum_{m=n+1-k}^{k-1}\frac{F_kF_m}{\sqrt{km}}\sqrt{k+m-n}.
	\end{equation}
	For $\om_1<n<\om_2$ one has $\Phi=\sqrt{m}$ hence 
	\begin{equation}
		\label{eq:31401<n<2}
		\mathcal{I}[g,\varphi_n\mathbbm{1}_{\{\varphi_{n,3}=1, \varphi_{n,4}=0\}\cap \{\om_1<n<\om_2\}}]=\frac{1}{\sqrt{n}}F_n\sum_{k=n+1}^{\infty}\sum_{m=1}^{n-1}\frac{F_kF_m}{\sqrt{k}}.
	\end{equation}
	When $n<\om_1<\om_2$ then $\Phi=\sqrt{n}$, from which we get 
	\begin{equation}
		\label{eq:3140n<1<2}
		\mathcal{I}[g,\varphi_n\mathbbm{1}_{\{\varphi_{n,3}=1, \varphi_{n,4}=0\}\cap \{n<\om_1<\om_2\}}]=F_n\sum_{k=n+2}^{\infty}\sum_{m=n+1}^{k-1}\frac{F_kF_m}{\sqrt{km}}.
	\end{equation}
	The three terms above are multiplied by a factor $2$ in \eqref{eq:dtFn} in view of the symmetry $\om_1 \leftrightarrow \om_2$.
	
	We are left only with the case $\om_1=\om_2$. If $\om_2>n$ then 
	\begin{equation}
		\label{eq:31401=2>n}
		\mathcal{I}[g,\varphi_n\mathbbm{1}_{\{\varphi_{n,3}=1, \varphi_{n,4}=0\}\cap \{\om_1=\om_2>n\}}]=F_n\sum_{k=n+1}^{\infty}\frac{F_k^2}{k}.
	\end{equation}
	When $\om_2<n$ one has
	\begin{equation}
		\label{eq:31401=2<n}
		\mathcal{I}[g,\varphi_n\mathbbm{1}_{\{\varphi_{n,3}=1, \varphi_{n,4}=0\}\cap \{\om_1=\om_2<n\}}]=\frac{F_n}{\sqrt{n}}\sum_{k=\lceil \frac{n}{2}\rceil }^{n-1}\frac{F_k^2}{k}\sqrt{2k-n}.
	\end{equation}

	\medskip $\bullet$ \textbf{Case $\varphi_{n,4}=1, \varphi_{n,3}=0$}
	In this case one has $\varphi_{n,1}=\varphi_{n,2}=0$. We assume $\om_1<\om_2$ as done previously. We again denote $k=\om_2 \text{ and } m=\om_1.$
	
	For $\om_1<\om_2<n$ then $\Phi=\sqrt{k+m-n}$ so that
	\begin{equation}
		\label{eq:41301<2<n}
		\mathcal{I}[g,\varphi_n\mathbbm{1}_{\{\varphi_{n,4}=1, \varphi_{n,3}=0\}\cap \{\om_1<\om_2<n\}}]=\sum_{k=2}^{n-1}\sum_{m=2}^{k-1}\frac{F_kF_mF_{k+m-n}}{\sqrt{kn}}.
	\end{equation}
	For $\om_1<n<\om_2$ one has $\Phi=\sqrt{m}$ hence 
	\begin{equation}
		\label{eq:41301<n<2}
		\mathcal{I}[g,\varphi_n\mathbbm{1}_{\{\varphi_{n,4}=1, \varphi_{n,3}=0\}\cap \{\om_1<n<\om_2\}}]=\sum_{k=n+1}^{\infty}\sum_{m=1}^{n-1}\frac{F_kF_mF_{k+m-n}}{\sqrt{k(k+m-n)}}.
	\end{equation}
	When $n<\om_1<\om_2$ then $\Phi=\sqrt{n}$, from which we get 
	\begin{equation}
		\label{eq:4130n<1<2}
		\mathcal{I}[g,\varphi_n\mathbbm{1}_{\{\varphi_{n,3}=1, \varphi_{n,4}=0\}\cap \{n<\om_1<\om_2\}}]=\sqrt{n}\sum_{k=n+2}^{\infty}\sum_{m=n+1}^{k-1}\frac{F_kF_mF_{k+m-n}}{\sqrt{km(k+m-n)}}.
	\end{equation}
	The three terms above are multiplied by a factor $2$ in \eqref{eq:dtFn} in view of the symmetry $\om_1 \leftrightarrow \om_2$.
	
	We are left only with the case $\om_1=\om_2$. If $\om_2>n$ then 
	\begin{equation}
		\label{eq:41301=2>n}
		\mathcal{I}[g,\varphi_n\mathbbm{1}_{\{\varphi_{n,4}=1, \varphi_{n,3}=0\}\cap \{\om_1=\om_2>n\}}]=\sqrt{n}\sum_{k=n+1}^{\infty}\frac{F_k^2F_{2k-n}}{k\sqrt{2k-n}}.
	\end{equation}
	When $\om_2<n$ we get
	\begin{equation}
		\label{eq:41301=2<n}
		\mathcal{I}[g,\varphi_n\mathbbm{1}_{\{\varphi_{n,4}=1, \varphi_{n,3}=0\}\cap \{\om_1=\om_2<n\}}]=\sum_{k=\lceil \frac{n+1}{2}\rceil}^{n-1}\frac{F_k^2F_{2k-n}}{k}.
	\end{equation}
	
	Therefore, from \eqref{eq:11213<n}-\eqref{eq:41301=2<n} the identity \eqref{eq:dtFn} follows. Notice the crucial cancellations of \eqref{eq:1120>3>n} with \eqref{eq:3140n<1<2} and \eqref{eq:1120>n3<n} with \eqref{eq:31401<n<2}. The proof of \eqref{eq:dtHn} immediately follows by the fact that $\de_t H_n=-F_n^{-2}\de_tF_n$.
	
\subsection*{Acknowledgements}

The authors are thankful to Juan J. L. Vel\'azquez for several helpful discussions and comments. The research of MD was supported by the Swiss State Secretariat for Education, Research and lnnovation (SERI) under contract number MB22.00034 through the project TENSE. The research of both authors was backed by the GNAMPA-INdAM.

\bibliographystyle{siam}
\bibliography{bibDWKE}
 \end{document}